\documentclass[10pt]{amsart}
\usepackage{graphicx}
\usepackage[active]{srcltx}
 \makeatletter
\renewcommand*\subjclass[2][2000]{%
  \def\@subjclass{#2}%
  \@ifundefined{subjclassname@#1}{%
    \ClassWarning{\@classname}{Unknown edition (#1) of Mathematics
      Subject Classification; using '1991'.}%
  }{%
    \@xp\let\@xp\subjclassname\csname subjclassname@#1\endcsname
  }%
}
 \makeatother
\usepackage{enumerate,url,amssymb,  mathrsfs,yhmath}

\newtheorem{theorem}{Theorem}[section]

\newtheorem*{lemma*}{Lemma}

\newtheorem{corollary}[theorem]{Corollary}

\theoremstyle{definition}

\newtheorem{example}[theorem]{Example}

\theoremstyle{remark}
\newtheorem{remark}[theorem]{Remark}

\numberwithin{equation}{section}

\newcommand{\abs}[1]{\lvert#1\rvert}

\def\XXint#1#2#3{{\setbox0=\hbox{$#1{#2#3}{\int}$}
\vcenter{\hbox{$#2#3$}}\kern-.5\wd0}}

\def\le{\leqslant}
\def\ge{\geqslant}
\begin{document}

\title{Gauss map of a harmonic surface} \subjclass{Primary 35J05; Secondary 47G10}


\author{David Kalaj}
\address{University of Montenegro, Faculty of Natural Sciences and
Mathematics, Cetinjski put b.b. 81000 Podgorica, Montenegro}
\email{davidk@t-com.me}

\begin{abstract} We prove that the distortion function of  the Gauss map of a
harmonic surface coincides with the distortion function  of the
surface. Consequently, Gauss map of a harmonic surface is
${\mathcal{K}}$ quasiconformal if and only if the surface is
${\mathcal{K}}$ quasiconformal, provided that the Gauss map is
regular or what is shown to be the same, provided that the surface
is non-planar.
\end{abstract} \maketitle


\section{Introduction and statement of the main result}
An orientation preserving  smooth mapping $\varphi :\Omega\to
\Omega'$, between two open domains $\Omega, \Omega'\subset \mathbf
C$ is called ${\mathcal{K}}$ (${\mathcal{K}}\ge 1$) quasiconformal
if dilatation $d_f$ of the Beltrami coefficient $\mu(z):={\varphi
_{\bar z}}/{\varphi _z}$ satisfies the inequality
$$d_\varphi :=\abs{\mu(z)}\le
k:=\frac{{\mathcal{K}}-1}{{\mathcal{K}}+1},$$ or what is the same if
$$\|\nabla \varphi\|^2\le
\frac 12\left({\mathcal{K}}+\frac
1{{\mathcal{K}}}\right)J_\varphi,$$ where $\|\nabla \varphi\|$ is
the Hilbert-Schmidt norm defined by $$ \|\nabla
\varphi\|^2:=|\varphi_z|^2+|\varphi_{\bar z}|^2$$ and $$J_\varphi
=(|\varphi_z|^2-|\varphi_{\bar z}|^2)$$ is the Jacobian of $\nabla
\varphi$.

 Note that in our context is enough to assume that $\varphi$ is smooth
mapping, however, the classical definition of quasiconfomality
assumes weaker conditions (cf. \cite[pp. 3, 23--24]{Ahl}). Note also
that, in this definition we do not assume injectivity.

\subsection{Parametric Surfaces}
We define an oriented parametric surface $\, \mathcal M\,$ in
$\,\mathbf R^3\,$ to be an equivalence class of mappings $\,Y = (a,
b , c) : D \rightarrow \mathbf R^3\,$ of some domain $\,D \subset
\mathbf C\,$ into $\,\mathbf R^3\,$, where the coordinate functions
$a,$ $b$, $c$ are of class at least $\,\mathscr C^1(D)\,$. Two such
mappings $X: D \rightarrow \mathbf R^3\,$ and $\,\tilde{X} :
\tilde{D} \rightarrow \mathbf R^3\,$, referred to as
parametrizations of the surface, are said to be equivalent if there
is a $\,\mathscr C^1$-diffeomorphism $\,\phi
:\tilde{D}\overset{\textnormal{\tiny{onto}}}{\longrightarrow} D \,$
of positive Jacobian determinant  such that $\; \tilde{Y} = Y\circ
\phi\,$.  Let us call such $\,\phi\,$ a \textit{change of variables,
or reparametrization} of the surface. Furthermore, we assume that
the branch (critical) points of  $\, \mathcal M\,$  are isolated.
These are the points  $\,(x,y) \in D\,$  at which the tangent
vectors $\,Y_x = \frac{\partial Y}{\partial x},\; Y_y =
\frac{\partial Y }{\partial y}\,$ are linearly dependent or
equivalently $Y_x\times Y_y\neq 0$ where by $\times$ is denoted the
standard vectorial product {in the space} $\Bbb R^3$. Equivalently,
at the critical points the \textit{Jacobian matrix} $\nabla Y(z)$
has rank at most 1. It has full rank  2 at the \textit{regular
points}. A surface with no critical points is called an
\textit{immersion} or a \textit{regular surface}.

For regular points of the surface we define the distortion function
by
$$\mathfrak{D}_Y=\frac{\|Y_x\|^2 + \|Y_{y}\|^2}{2\|Y_x\times
Y_y\|}.$$

A mapping $Y$ is called ${\mathcal{K}}$ quasiconformal if
\begin{equation}\label{surfeq}\|Y_x\|^2 + \|Y_{y}\|^2 \le
\left({\mathcal{K}}+\frac{1}{{\mathcal{K}}}\right)\|Y_x\times
Y_y\|,\ \ z=x+iy\in D.\end{equation} If ${\mathcal{K}} = 1$ then
\eqref{surfeq} is equivalent to the system of the equations

\begin{equation}\label{surfeq1}
\|Y_x\|^2= \|Y_y\|^2\ \ \text{and} \ \ \left<Y_x,Y_y\right> = 0,
\end{equation}
which represent isothermal (conformal) coordinates of the surface
$M$. If $u$ is harmonic and satisfies the system \eqref{surfeq1}
then $M$ is a minimal surface.

If $Y$ is an immersion of $M$ and $X=(u,v,w)$ are isothermal
coordinates of a smooth surface $M$, and if $\varphi=X^{-1}\circ Y$,
then we have
\begin{equation}\label{frak}\mathfrak{D}_Y=\mathfrak{D}_\varphi=\frac{|\varphi
_z|^2+|\varphi_{\bar z}|^2}{|\varphi_z|^2-|\varphi_{\bar
z}|^2}.\end{equation}
\subsection{Gauss map of a surface} Let $X:\Omega\to \mathbf R^3$ be a smooth regular surface and let $M=X(\Omega)$. Let $S^2$ be the unit sphere in $\mathbf R^3$
and let $\mathbf{n} : \Omega\to  S^2$ be the orientation preserving
Gauss map of $M$ defined as follows
$$\mathbf{n}(z)=\frac{X_u\times X_v}{\|X_u\times X_v\|}.$$ If we denote by $f : S^2\setminus \{(0, 0, 1)\} \to \mathbf
R^2$, $$f(x_1,x_2,x_3)= \left(\frac{x_1}{1- x_3}, \frac{x_2}{1-
x_3}\right)$$ the stereographic projection, then the orientation
preserving map $\mathfrak{g} := f\circ \mathbf{n}$ is said to be the
complex Gauss map of $X$.

Since $f$ is a conformal mapping, then $\mathfrak{g}$ is
quasiconformal if and only if $\mathbf{n}$ is quasiconformal with
the same constant of quasiregularity. Moreover
\begin{equation}\label{math}\mathfrak{D}_{\mathbf n}=\mathfrak{D}_{\mathfrak
g}=\frac{|\mathfrak g_z|^2+|\mathfrak g_{\bar z}|^2}{|\mathfrak
g_z|^2-|\mathfrak g_{\bar z}|^2}.\end{equation}

A smooth enough surface can be parameterized using a isothermal
parameterization. Such a parameterization is minimal if the
coordinate functions $x_k$ are harmonic, i.e., $\phi_k(\zeta)$ are
analytic. A minimal surface can therefore be defined by a triple of
analytic functions such that
\begin{equation}\label{2}\phi_1^2+\phi_2^2+\phi_3^2=0. \end{equation}

The real parameterization is then obtained as
\begin{equation}\label{3}x_k=\Re \int \phi_k(\zeta)d\zeta.
\end{equation}

But, for an analytic function $p$ and a meromorphic function $q$,
the triple of functions \begin{equation}\label{31}\phi_1(\zeta) =
p(1-q^2)\end{equation}
  \begin{equation}\label{4} \phi_2(\zeta) = ip(1+q^2)
\end{equation}

\begin{equation}\label{5} \phi_3(\zeta) =   2pq \end{equation}
are analytic as long as $p$ has a zero of order $\ge m$ at every
pole of $q$ of order $m$. This gives a minimal surface in terms of
the Enneper-Weierstrass parameterization
\begin{equation}\label{7} \Re \int\{p(1-q^2); ip(1+q^2);
2pq\}d\zeta.\end{equation}

It is well known that if the surface is minimal endowed with
Enneper-Weierstrass parameterization, then its complex Gauss map is
a meromorphic function $\mathfrak{g}(w)=-i/q'(w)$. See for example
\cite[Sec.~9.3]{dure} for the above facts.

In this paper, we extend this result to quasiconformal harmonic
surfaces by proving the following theorem.

\begin{theorem}\label{qaz}
The dilatation and distortion function of the Gauss map $\mathbf n$
of a harmonic surface $X$ coincides with the dilatation and
distortion function of the surface itself, provided that the Gauss
map is regular. In other words, if
\begin{equation}\label{kusht}\frac{\partial \mathbf{n}}{\partial x}(z)\times
\frac{\partial \mathbf{n}}{\partial y}(z)\neq 0,\end{equation} then
\begin{equation}\label{per1}\mathfrak{D}_{\mathbf n}(z)=\mathfrak{D}_{\mathbf X}(z)\end{equation} and
\begin{equation}\label{per2}d_X(z)=d_{\mathbf n}(z).\end{equation}
\end{theorem}

\begin{corollary}\label{qazu}
If  a harmonic parametric surface $X$ is ${\mathcal{K}}$
quasiconformal then its Gauss map is ${\mathcal{K}}$ quasiconformal,
privided that the Gauss map is regular.
\end{corollary}

We also prove the following theorem
\begin{theorem}\label{planar}
The Gauss map of a harmonic surface is regular if and only if the
surface is not planar. If the Gauss map of a surface is not regular,
then it must be a constant vector.
\end{theorem}

\begin{remark}
From Theorem~\ref{planar} we find out that, in Theorem~\ref{qaz},
instead of the condition "the Gauss map is regular" we can simply
say "the surface is non-planar". When we say that the Gauss map is
regular we mean that, the cross product
$\frac{\partial}{dx}\mathbf{n}\times
 \frac{\partial}{dy}\mathbf{n}$ is not identically zero at some open
 subset of the domain. However, as it can be shown by the example $$X=\left\{x,-\frac{x^3}{3}+x
 \left(\frac{1}{2}+y\right)^2,1-x^2+y+y^2\right\},$$ that $$\frac{\partial}{dx}\mathbf{n}\times
 \frac{\partial}{dy}\mathbf{n}=0,$$ for $y=-1/2$. Thus, the branch
 points of the Gauss normal of harmonic surfaces are not isolated, as in the case of
 minimal surfaces.
\end{remark}

\begin{remark}
The class of minimal surfaces with Enneper-Weierstrass
parameterization is a special case of the class of quasiconformal
harmonic surfaces. Namely in this case $\mathcal{K}=1$, because the
coordinates are isothermal (conformal). In this case the condition
\eqref{kusht} reduces to the condition that
$\mathfrak{g}(w)=-i/q'(w)=const$, i.e. the Gauss normal is a
constant. This implies that the minimal surface is planar.

The class of quasiconformal harmonic mapping between complex domains
and two-dimensional surfaces has been very active research of
investigation in recent years. For some regularity results of this
class we refer to the following papers \cite{kalajan}, \cite{mathz},
\cite{km}, \cite{MP}. For some regularity results of the class of
minimal surfaces we refer to the papers of J. C. C. Nitsche
\cite{Ni} and \cite{nic}.

Recall that by a definition of A. Alarcon and F. J. Lopez
\cite{harsur} a harmonic immersion $X : M \to \mathbf R^3$ is said
to be quasiconformal ($QC$ for short) if its orientation preserving
Gauss map $\mathbf{n} : M \to S^2$ is quasiconformal (or
equivalently, if $\mathfrak{g}$ is quasiconformal). Among other
special features, quasiconformal harmonic immersions are
quasiminimal in the sense of Osserman \cite{or}. In this case, $X$
is said to be a harmonic $QC$ parameterization of the harmonic
surface $X(M)$. Notice also this, a harmonic surface has no elliptic
points by the maximum principle for harmonic functions. In other
words, its Gauss curvature is nowhere positive. Let $w$ a harmonic
diffeomorphism of the unit disk onto itself which is not
quasiconformal. Let $X(z)=(\Re (w),\Im (w),0)$. Then
$\mathbf{n}=(0,0,1)$ and therefore it is a $1-$quasiconformal
mapping. This mean that the fact that the condition "the Gauss map
is regular" is important in Theorem~\ref{qaz}. In other words, two
above definitions of quasiconformality are equivalent, provided that
the Gauss map is regular (up to branch points) or what is the same,
if the surface is not planar.
\end{remark}
\section{Proofs}
\begin{proof}[Proof of Theorem~\ref{qaz}]  Let $$X(x,y)=(a(x,y),b(x,y),c(x,y)).$$ Without loos of
generality we can assume that $a(z)=x$. Namely, let $\phi$ be a
analytic mapping of the unit disk into $\mathbf C$ such that
$$\phi(z)=a(z)+i \tilde a(z).$$ Here $\tilde a(z)$ is the harmonic conjugate of $a(z)$. Take  instead of
$$X(z)=(a(z),b(z),c(z))$$ $$\tilde X = X\circ \phi^{-1}(z)$$ in some
neighborhood $D$ of some nonsingular point $z$ of $\phi'(z)$. Then
the first coordinate of $\tilde f$ is $x$ in $D$.

Let $$\mathbf{n}(z)=\frac{X_x\times X_y}{\|X_x\times X_y\|}.$$ Then
$\mathbf{n}(z)$ is given by
$$\left\{\frac{c_y b_x-b_y c_x}{\sqrt{b_y^2+c_y^2+\left(c_y b_x-b_y
c_x\right)^2}},-\frac{c_y}{\sqrt{b_y^2+c_y^2+\left(c_y b_x-b_y
c_x\right)^2}},\frac{b_y}{\sqrt{b_y^2+c_y^2+\left(c_y b_x-b_y
c_x\right)^2}}\right\}.$$ Define
$$P=\frac{d}{dx}\mathbf{n}(z), \ \ \ Q=\frac{d}{dy}\mathbf{n}(z).$$
The distortion function is $$\mathfrak{D}_{\mathbf
n}=\frac{|P|^2+|Q|^2}{|P\times Q|}.$$ Then
\begin{equation}\label{epara}|P|^2+|Q|^2=\frac{N}{G^4}\end{equation} and
\begin{equation}\label{edyta}|P\times Q|=\frac{M}{G^3}\end{equation}
where
$$M=\left(c_{y}^2 \left(b_{xy}^2-b_{yy} b_{xx}\right)+b_{y}
c_{y} \left(-2 b_{xy} c_{xy}+c_{yy} b_{xx}+b_{yy}
c_{xx}\right)+b_{y}^2 \left(c_{xy}^2-c_{yy} c_{xx}\right)\right),$$

$$G=\sqrt{{b_y}^2+{c_{y}}^2+\left({c_{y}} {b_x}-{b_y}
{c_{x}}\right)^2}$$

 and
\[\begin{split}N&= c_y^4 \left(b_{xy}^2+b_{xx}^2\right) \\&+b_y^2 \left(c_{yy}^2
\left(1+b_x^2+c_x^2\right)-2 b_y c_{yy} b_x
c_{xy}+\left(1+b_y^2+b_x^2+c_x^2\right) c_{xy}^2-2 b_y b_x c_{xy}
c_{xx}+b_y^2 c_{xx}^2\right)\\&-2 c_y^3 \left(b_{yy} c_x b_{xy}+c_x
b_{xy} b_{xx}+b_y \left(b_{xy} c_{xy}+b_{xx}
c_{xx}\right)\right)\\
&-2 b_y c_y \bigg[\left(1+b_x^2+c_x^2\right) b_{xy} c_{xy}+b_{yy}
\left(c_{yy} \left(1+b_x^2+c_x^2\right)-b_y b_x c_{xy}\right)\\&
+b_y^2 \left(b_{xy} c_{xy}+b_{xx} c_{xx}\right)-b_y \left(c_{yy}
\left(b_x b_{xy}-c_x c_{xy}\right)-c_x c_{xy} c_{xx}+b_x
\left(c_{xy} b_{xx}+b_{xy} c_{xx}\right)\right)\bigg]\\&+ c_y^2
\bigg[b_{yy}^2 \left(1+b_x^2+c_x^2\right)+\left(1+b_x^2+c_x^2\right)
b_{xy}^2+2 b_y b_{yy} \left(-b_x b_{xy}+c_x c_{xy}\right)\\& +b_y^2
\left(b_{xy}^2+c_{xy}^2+b_{xx}^2+c_{xx}^2\right)+2 b_y \left(c_{yy}
c_x b_{xy}-b_x b_{xy} b_{xx}+c_x \left(c_{xy} b_{xx}+b_{xy}
c_{xx}\right)\right)\bigg].\end{split}\] Let
$$A=b_{xy}^2-b_{yy} b_{xx},\ \ \ B=b_{xy} c_{xy}+b_{xx} c_{xx},\ \ \ C=c_{xy}^2-c_{yy} c_{xx}$$
$$\delta =\left(1+b_x^2+c_x^2\right).$$ Then, since $\Delta b =
\Delta c=0$ we obtain $$M=\left(c_{y}^2 A-2b_{y} c_{y} B+b_{y}^2
C\right)$$ and
\[\begin{split}N&= c_y^4 A \\&+b_y^2 \left(c_{yy}^2
\delta-2 b_y c_{yy} b_x c_{xy}+\left(1+b_y^2+b_x^2+c_x^2\right)
c_{xy}^2-2 b_y b_x c_{xy} c_{xx}+b_y^2 c_{xx}^2\right)\\&-2 c_y^3
\left(b_{yy} c_x b_{xy}+c_x b_{xy} b_{xx}+b_y B)\right)\\
&-2 b_y c_y \bigg[(\delta+b_y^2) B-b_y b_xb_{yy}  c_{xy} -b_y
\left(c_{yy} \left(b_x b_{xy}-c_x c_{xy}\right)-c_x c_{xy}
c_{xx}+b_x \left(c_{xy} b_{xx}+b_{xy} c_{xx}\right)\right)\bigg]\\&+
c_y^2 \bigg[\delta A+2 b_y b_{yy} \left(-b_x b_{xy}+c_x
c_{xy}\right)\\& +b_y^2 \left(A+C\right)+2 b_y \left(c_{yy} c_x
b_{xy}-b_x b_{xy} b_{xx}+c_x \left(c_{xy} b_{xx}+b_{xy}
c_{xx}\right)\right)\bigg].\end{split}\] Further, since $\Delta b=0$
\[\begin{split}N&= c_y^4 A \\&+b_y^2 \left(c_{yy}^2
\delta-2 b_y c_{yy} b_x c_{xy}+\left(1+b_y^2+b_x^2+c_x^2\right)
c_{xy}^2-2 b_y b_x c_{xy} c_{xx}+b_y^2 c_{xx}^2\right)-2 c_y^3 b_y B\\
&-2 b_y c_y \bigg[(\delta+b_y^2) B-b_y b_xb_{yy}  c_{xy} -b_y
\left(c_{yy} \left(b_x b_{xy}-c_x c_{xy}\right)-c_x c_{xy}
c_{xx}+b_x \left(c_{xy} b_{xx}+b_{xy} c_{xx}\right)\right)\bigg]\\&+
c_y^2 \bigg[\delta A+2 b_y b_{yy} \left(c_x c_{xy}\right) +b_y^2
\left(A+C\right)+2 b_y \left(c_{yy} c_x b_{xy}+c_x \left(c_{xy}
b_{xx}+b_{xy} c_{xx}\right)\right)\bigg]\end{split}\] and

\[\begin{split}N&= c_y^4 A +b_y^2 \left(c_{yy}^2
\delta+\left(1+b_y^2+b_x^2+c_x^2\right)
c_{xy}^2+b_y^2 c_{xx}^2\right)-2 c_y^3 b_y B\\
&-2 b_y c_y \bigg[(\delta+b_y^2) B-b_y b_xb_{yy}  c_{xy} -b_y
\left(c_{yy} \left(b_x b_{xy}-c_x c_{xy}\right)-c_x c_{xy}
c_{xx}+b_x \left(c_{xy} b_{xx}+b_{xy} c_{xx}\right)\right)\bigg]\\&+
c_y^2 \bigg[\delta A+2 b_y b_{yy} \left(c_x c_{xy}\right) +b_y^2
\left(A+C\right)+2 b_y \left(c_{yy} c_x b_{xy}+c_x \left(c_{xy}
b_{xx}+b_{xy} c_{xx}\right)\right)\bigg].\end{split}\] Finally

\[\begin{split}N&= c_y^4 A +b_y^2 \left(c_{yy}^2
\delta+\left(1+b_y^2+b_x^2+c_x^2\right)
c_{xy}^2+b_y^2 c_{xx}^2\right)-2 c_y^3 b_y B\\
&-2 b_y c_y \bigg[(\delta+b_y^2) B\bigg]+ c_y^2 \bigg[\delta A
+b_y^2 \left(A+C\right)\bigg]\\&= c_y^4 A +b_y^2
\left((b_y^2+\delta) C\right)-2 c_y^3 b_y B-2 b_y c_y
\bigg[(\delta+b_y^2) B\bigg]+ c_y^2 \bigg[\delta A +b_y^2
\left(A+C\right)\bigg]\\&=c_y^2(c_y^2+\delta+b_y^2)
A+b_y^2(b_y^2+\delta+c_y^2)C-2b_yc_y(\delta+b_y^2+c_y^2)C \\&=\gamma
\left(c_y^2 A+b_y^2C-2b_yc_y C\right)\end{split}\] where
$$\gamma=1+b_x^2+b_y^2+c_x^2+c_y^2.$$
As $$\mathfrak{D}_{\mathbf n}=\frac{N}{GM}$$ we obtain
\[\begin{split}\mathfrak{D}_{\mathbf n}&=\frac{\gamma}{G}\frac{\left(c_y^2 A+b_y^2C-2b_yc_y C\right)}{\left(c_y^2 A+b_y^2C-2b_yc_y C\right)}
\\&=\frac{1+b_y^2+c_y^2+b_x^2+c_x^2}{\sqrt{
b^2_y+c^2_y+\left(c_y b_x-b_y c_x\right)^2}}=\mathfrak{D}_{\mathbf
X}.\end{split}\] The last part of the theorem follows from the
formulas \eqref{frak} and \eqref{math}.
\end{proof}

\begin{corollary}[Berenstein theorem for noparametric harmonic surfaces] If the harmonic noparametric surface over $\mathbf R^2$ is
quasiconformal, then the surface is planar.
\end{corollary}
\begin{proof}From the previous
theorem we find out that the Gauss map is quasiconformal. Then by a
theorem of L. Simon \cite{simon}, $X$ must be planar.
\end{proof}

\begin{corollary}
If $X=(a,b,c)$ is a harmonic surface then $$\mathfrak{D}_{\mathbf
n}=\frac{|\nabla a|^2+|\nabla b|^2+|\nabla c|^2}{2\sqrt{\left(b_y
a_x-a_y b_x\right)^2+\left(-c_y a_x+a_y c_x\right)^2+\left(c_y
b_x-b_y c_x\right)^2}}.$$
\end{corollary}

\begin{proof}[Proof of Theorem~\ref{planar}] As in the proof of Theorem~\ref{qaz}, we can assume
that the first coordinate of the surface is $x$. Further, by using
the same notation
$$A=b_{xy}^2-b_{yy} b_{xx},\ \ \ B=b_{xy} c_{xy}+b_{xx} c_{xx},\ \ \ C=c_{xy}^2-c_{yy} c_{xx}$$
$$
G=\sqrt{{b_y}^2+{c_{y}}^2+\left({c_{y}} {b_x}-{b_y}
{c_{x}}\right)^2},$$ and
 $$M=\left(c_{y}^2 A-2b_{y} c_{y} B+b_{y}^2 C\right),$$ we obtain
 that  $$\frac{\partial}{dx}\mathbf{n}\times
 \frac{\partial}{dy}\mathbf{n}=0$$ if and only if $M=0$. The last is
 equivalent with $$(|c_y|\sqrt{A}-|b_y|\sqrt{C})^2+2(|b_y||c_y|\sqrt{AC}-b_yc_y
C)=0.$$ It is an elementary application of Cauchy-Schwarz inequality
that $$\sqrt{AC}\ge B.$$ Moreover $$AC=B^2$$ if and only if $$b_{xx}
c_{xx}=b_{xy}c_{xy}.$$ Thus $M=0$ if and only if
$$(|c_y|\sqrt{A}-|b_y|\sqrt{C})=0$$ and $$|b_y||c_y|\sqrt{AC}-b_yc_y
C=0.$$ Thus $$b_y\ge 0,\ \ \ c_y\ge 0,$$ $$b_{xx}
c_{xx}=b_{xy}c_{xy}$$ and $$c_y\sqrt{A}-b_y\sqrt{C}=0.$$ The cases
$b_y=0$ or $c_y=0$ are trivial. So assume that $b_y\neq 0$ and
$c_y\neq 0$. Combining the last two equalities we arrive at equality
$$c_y b_{yy}=b_y c_{yy}.$$
It follows that $$\frac{d}{dy}\frac{b_y}{c_y}=0$$ i.e.
$$\frac{b_y}{c_y}=\lambda(x),$$ for some real function $\lambda(x)$
depending only on $x$. Since $b_y$ and $c_y$ are real analytic, it
follows that $\lambda$ is a real analytic function, i.e.
$$\lambda(x)=\sum_{n=0}^\infty\lambda_n x^n.$$ Further
$$c_y = \sum_{n=0}^\infty (a_nz^n+b_n \bar z^n)$$ and therefore
$$b_y(z) =\sum_{n=0}^n\lambda_n x^n\sum_{n=0}^\infty (a_nz^n+b_n \bar
z^n).$$ Let us show that $\lambda_n=0$ for $n\ge 2$.

We will use the following well-known fact. Any harmonic function has
(locally) a unique representation as a sum of homogeneous harmonic
polynomials $\alpha_n z^n +\beta_n \overline z^n$. Since $\lambda_1
x (a_1 z + b_1 \bar z)$ is the only possible harmonic polinom of
degree $2$ in the expression for $b_y$, it must be $a_1 z + b_1 \bar
z=2a_1y$. Further
$$\sum_{k=0}^{n-1} \lambda_{n-k} x^{n-k} (a_k z^k +b_k \bar z^k)$$
is a harmonic polinom of degree $n$ and is therefore equal to
$$\alpha_n z^n +\beta_n \bar z^n.$$ Thus $$\sum_{k=0}^{n-1}
\lambda_{n-k} (z+\bar z)^{n-k} (a'_k z^k +b'_k \bar z^k)=\alpha_n
z^n +\beta_n \bar z^n.$$ Since the representation $$\sum_{i,j}
q_{ij} z^i \bar z^j$$ is unique, it follows that for $n\ge 2$,
$\lambda_n=0$. Hence
$$\frac{b_y}{c_y}=\lambda_1x +\lambda_0.$$ By a similar argument we
find that, if $\lambda_1\neq 0$, then $c_y=\omega y +\nu$. In this
case
$$c(x,y)=\frac{\omega}{2}(y^2-x^2)+\nu y+\nu_1 x +\nu_0$$ and $$b_y=(\lambda_1x
+\lambda_0)(\omega y+\nu).$$ Thus
$$b(x,y)=\frac{\omega}{2}(\lambda_1x
+\lambda_0)(y+\nu/\omega)^2-\frac{\omega}{6\lambda_1^2}(\lambda_1x
+\lambda_0)^3.$$ However $b$ and $c$ obtained in this case do not
satisfy the equation $M\equiv 0$ in an open set. Namely
$$M=\lambda_1^4(\nu +\omega y)^8\not\equiv 0.$$ It remains to consider the
case $\lambda_1=0$. Then
$$b(z)=\lambda_0 c(z)+\nu(x),$$ implying that $\nu(x)=\nu_0+\nu_1
x$. Thus
$$X(x)=(x,\lambda_0 c(z)+\nu_0+\nu_1 x, c(z)).$$ Thus the Gauss map
of the surface $X$ is $$
\left\{\frac{\nu_1}{\sqrt{2+\nu_1^2}},-\frac{1}{\sqrt{2+\nu_1^2}},\frac{1}{\sqrt{2+\nu_1^2}}\right\}$$
implying that the surface is planar.

\end{proof}
\subsection{An open problem} Whether Berenstein theorem is true for
quasiconformal parametric harmonic surfaces?

\bibliographystyle{amsplain}
\providecommand{\bysame}{\leavevmode\hbox
to3em{\hrulefill}\thinspace}
\providecommand{\MR}{\relax\ifhmode\unskip\space\fi MR }
\providecommand{\MRhref}[2]{%
  \href{http://www.ams.org/mathscinet-getitem?mr=#1}{#2}
} \providecommand{\href}[2]{#2}

\end{document}